\documentclass[12pt,a4paper]{article}
\usepackage{a4wide,amsfonts,amsmath,latexsym,amssymb,euscript,eufrak,graphicx,units,mathrsfs}
\usepackage[utf8]{inputenc}
\usepackage{amsmath}
\usepackage{amsfonts}
\usepackage{amssymb}
\usepackage{amsthm}
\usepackage{floatrow}
\usepackage{blindtext}

\usepackage[usenames,dvipsnames]{color}

\numberwithin{equation}{section}

\newtheorem{theorem}{Theorem}[section]

\newtheorem{lemma}{Lemma}[section]

\newtheorem{definition}{Definition}[section]

\begin{document}

\renewcommand{\thefootnote}{\arabic{footnote}}

\begin{center}
{\Large \textbf{Least squares estimation for non-ergodic weighted fractional Ornstein-Uhlenbeck process of general parameters}} \\[0pt]
~\\[0pt]
Abdulaziz Alsenafi\footnote{ Department of Mathematics, Faculty of
Science, Kuwait University, Kuwait. E-mail:
 \texttt{abdulaziz.alsenafi@ku.edu.kw}\\ *Corresponding author}* \ Mishari
Al-Foraih\footnote{ Department of Mathematics, Faculty of Science,
Kuwait University, Kuwait. E-mail:
 \texttt{mishari.alforaih@ku.edu.kw}} \
 Khalifa Es-Sebaiy\footnote{ Department of Mathematics, Faculty of
Science, Kuwait University, Kuwait. E-mail:
 \texttt{khalifa.essebaiy@ku.edu.kw}}
\\[0pt]
~\\[0pt]
 Kuwait University \\[0pt]
~\\[0pt]
\end{center}

\begin{abstract}
Let $B^{a,b}:=\{B_t^{a,b},t\geq0\}$ be a weighted fractional
Brownian motion  of parameters  $a>-1$, $|b|<1$, $|b|<a+1$. We
consider a  least square-type method to estimate the drift parameter
$\theta>0$   of the weighted fractional Ornstein-Uhlenbeck process
$X:=\{X_t,t\geq0\}$ defined by $X_0=0; \ dX_t=\theta
X_tdt+dB_t^{a,b}$.
 In this work, we provide   least squares-type estimators for
$\theta$ based continuous-time and discrete-time observations of
$X$. The strong consistency and the asymptotic behavior in
distribution of the estimators  are studied for all $(a,b)$ such
that $a>-1$, $|b|<1$, $|b|<a+1$. Here we extend the results of
\cite{SYY,SYY2} (resp. \cite{CSC}), where the strong consistency and
the asymptotic distribution of the estimators are proved for
 $-\frac12<a<0$,   $-a<b<a+1$ (resp. $-1<a<0$, $-a<b<a+1$).  Simulations
are performed to illustrate the theoretical results.
\end{abstract}

{\noindent }\textbf{Key words}: Drift parameter estimation; Weighted
fractional Ornstein-Uhlenbeck process; Strong consistency;
Asymptotic distribution.\\


\section{Introduction}
Parameter estimation for non-ergodic type diffusion processes has
been developed in several papers. For motivation and further
references, we refer the reader to Basawa and Scott \cite{BS}, Dietz
and Kutoyants \cite{DK}, Jacod \cite{jacod} and Shimizu
\cite{shimizu}.

Let $B^{a,b}:=\{B_t^{a,b},t\geq0\}$ be a weighted fractional
Brownian motion (wfBm) with parameters $(a,b)$ such that $a>-1$,
$|b|<1$ and  $|b|<a+1$, that is, $B^{a,b}$ is defined as a centered
Gaussian process starting from zero with covariance
\begin{eqnarray}R^{a,b}(t,s)=E\left(B_t^{a,b}B_s^{a,b}\right)=\int_0^{s\wedge
t} u^a\left[(t-u)^b+(s-u)^b\right]du,\quad s,t\geq0.\label{cov-B}
\end{eqnarray}
For $a = 0$, $-1 < b < 1$, the wfBm is a fractional Brownian motion
(fBm). The process $B^{a,b}$ was introduced by \cite{BGT} as an
extension
 of  fBm. Moreover, it shares several properties with
fBm, such as   self-similarity, path continuity, behavior of
increments, long-range dependence, non-semimartingale, and others.
But, unlike fBm, the wfBm  does not have stationary increments for
$a\neq0$. For more details about the subject, we refer the reader to
\cite{BGT}.

In this work we consider the non-ergodic Ornstein-Uhlenbeck process
$X:=\{X_t,t\geq0\}$ driven by a wfBm $B^{a,b}$, that is the unique
solution of the following linear stochastic differential equation
\begin{eqnarray}X_0=0; \qquad dX_t=\theta X_tdt+dB_t^{a,b},\label{wfBm}
\end{eqnarray}
where $\theta>0$ is an unknown parameter.

An example of interesting problem related to (\ref{wfBm}) is the
statistical estimation of  $\theta$ when one observes $X$. In recent
years, several researchers have been interested in studying
statistical estimation problems for Gaussian Ornstein-Uhlenbeck
processes. Let us mention some works in this direction in this case
of  Ornstein-Uhlenbeck process driven by a fractional Brownian
motion $B^{0,b}$, that is, the solution of (\ref{wfBm}), where
$a=0$. In the ergodic case corresponding to $\theta<0$, the
statistical estimation for the parameter  $\theta$
 has been studied by several papers, for
instance  \cite{HN,EEV,HNZ,EV,DEV} and the references therein.
Further, in the non-ergodic case corresponding to $\theta>0$, the
estimation of $\theta$ has been considered by using least squares
method, for example in \cite{EEO,BEO,EAA,EN} and the references
therein.

Here our aim is to estimate  the the drift parameter $\theta$ based
on continuous-time and discrete-time observations of $X$, by using
least squares-type estimators (LSEs) for $\theta$.\\
First we will consider the following LSE
 \begin{eqnarray}\label{estimator-cont} \widetilde{\theta}_t =\frac{X^2_t}{2\int_0^tX_s^2ds},\quad
 t\geq0,
\end{eqnarray}
as statistic to estimate $\theta$ based on the continuous-time
observations $\{X_s,\ s\in[0,t]\}$ of (\ref{wfBm}), as
$t\rightarrow\infty$. We will prove the strong consistency and the
asymptotic behavior in distribution of the estimator
$\widetilde{\theta}_t$   for all parameters $a>-1$, $|b|<1$ and
$|b|<a+1$. Our results extend those proved in \cite{SYY,SYY2}, where
$-\frac12<a<0$,   $-a<b<a+1$ only.

Further, from a practical point of view, in parametric inference, it
is more realistic and interesting to consider asymptotic estimation
for (\ref{wfBm}) based on discrete observations. So, we will assume
that the process $X$ given in (\ref{wfBm}) is observed equidistantly
in time with the step size $\Delta_n$: $t_i=i\Delta_{n},
i=0,\ldots,n$, and $T_n=n\Delta_{n}$ denotes the length of
the``observation window". Then we will consider the following
estimators
\begin{eqnarray}\label{expression thetahat}
\hat{\theta}_{n}=\displaystyle\frac{\displaystyle\sum_{i=1}^{n}X_{t_{i-1}}(X_{t_{i}}-X_{t_{i-1}})}{\Delta_n
\displaystyle\sum_{i=1}^{n}X_{t_{i-1}}^{2}},
\end{eqnarray}
and
\begin{eqnarray}\label{expression thetaCheck}
\check{\theta}_n=\displaystyle \frac{X_{T_{n}}^{2}}{2\Delta_n
\displaystyle\sum_{i=1}^{n}X_{t_{i-1}}^{2}}
\end{eqnarray}
as statistics to estimate $\theta$ based on the sampling data
$X_{t_i}, i=0,\ldots,n$, as $\Delta_n\longrightarrow0$ and
$n\longrightarrow\infty$. We will study the asymptotic behavior and
the rate consistency of the estimators $\hat{\theta}_{n}$ and
$\check{\theta}_n$ for all parameters $a>-1$, $|b|<1$ and $|b|<a+1$.
In this case, our results extend those proved in \cite{CSC}, where
$-1<a<0$, $-a<b<a+1$ only.

The rest of the paper is organized as follows. In Section 2, we
present auxiliary results that are used in the calculations of the
paper.  In Section 3, we prove the consistency and the asymptotic
distribution of the  estimator $\widetilde{\theta}_t$ given in
(\ref{estimator-cont}), based on the continuous-time observations of
$X$.  In Section 3, we  study the asymptotic behavior and the rate
consistency of the estimators $\hat{\theta}_{n}$ and
$\check{\theta}_n$ defined in (\ref{expression thetahat}) and
(\ref{expression thetaCheck}), respectively, based on the
discrete-time observations of $X$. Our theoretical study is
completed with simulations. We end the paper with a short review on
some results from \cite{EEO,EAA} needed for the proofs of our
results.

\section{Auxiliary Results}
 This section is devoted to prove some technical  ingredients, which will be needed
throughout this paper.\\ In the following lemma we provide a useful
decomposition of the covariance function $R^{a,b}(t,s)$ of
$B^{a,b}$.
\begin{lemma}\label{lemma-cov-R}Suppose that  $a>-1$,
$|b|<1$ and  $|b|<a+1$.   Then we can rewrite the covariance
$R^{a,b}(t,s)$ of  $B^{a,b}$, given in (\ref{decomp-main-R})  as
\begin{eqnarray}R^{a,b}(t,s)
&=&\beta\left(a+1,b+1\right)\left[t^{a+b+1}+s^{a+b+1}\right]-m(t,s),
\label{decomp-main-R}
\end{eqnarray}
where $\beta(c, d) = \int_0^1 x^{c-1}(1-x)^{d-1}$ denotes the usual
beta function, and the function $m(t,s)$ is defined by
\begin{eqnarray}m(t,s)&:=& \int_{s\wedge t}^{s\vee t} u^a(t\vee
s-u)^bdu. \label{function-m}
\end{eqnarray}
\end{lemma}
\begin{proof}
We have for every $s,t\geq0$,
\begin{eqnarray}&&R^{a,b}(t,s)\\&=&E\left(B_t^{a,b}B_s^{a,b}\right)\nonumber
\\&=&\int_0^{s\wedge
t} u^a\left[(t-u)^b+(s-u)^b\right]du\nonumber
\\&=&\int_0^{s\wedge
t} u^a\left[(t\vee s-u)^b+(t\wedge s-u)^b\right]du\nonumber
\\&=&\int_0^{s\wedge
t} u^a(t\vee s-u)^bdu+\int_0^{s\wedge t}u^a(t\wedge
s-u)^bdu\nonumber
\\&=&\int_0^{s\vee
t} u^a(t\vee s-u)^bdu-\int_{s\wedge t}^{s\vee t} u^a(t\vee
s-u)^bdu+\int_0^{s\wedge t}u^a(t\wedge s-u)^bdu.\label{decomp1-R}
\end{eqnarray}
Further, making change of variables $x=u/t$, we have for every
$t\geq0$,
\begin{eqnarray} \int_0^{t} u^a (t-u)^b du&=&t^b\int_0^{t} u^a \left(1-\frac{u}{t}\right)^b du\nonumber
\\&=&t^{a+b+1}\int_0^{1} x^a \left(1-x\right)^b du\nonumber
\\&=&t^{a+b+1}\beta\left(a+1,b+1\right).\label{decomp2-R}
\end{eqnarray}\\
Therefore, combining (\ref{decomp1-R}) and  (\ref{decomp2-R}), we
deduce that
\begin{eqnarray}R^{a,b}(t,s)&=&\beta\left(a+1,b+1\right)\left[(t\vee s)^{a+b+1}+(t\wedge s)^{a+b+1}\right]-\int_{s\wedge t}^{s\vee
t} u^a(t\vee s-u)^bdu\nonumber\\
&=&\beta\left(a+1,b+1\right)\left[t^{a+b+1}+s^{a+b+1}\right]-\int_{s\wedge
t}^{s\vee t} u^a(t\vee s-u)^bdu, \label{decomp3-R}
\end{eqnarray}
which proves (\ref{decomp-main-R}).
\end{proof}

We will also need the following technical lemma.
\begin{lemma}We have as $t\longrightarrow\infty$,
\begin{eqnarray}I_t:=t^{-a}e^{-\theta t}\int_0^te^{\theta
s}m(t,s)ds\longrightarrow\frac{\Gamma(b+1)}{\theta^{b+2}},\label{cv1-m}
\end{eqnarray}
\begin{eqnarray}J_t:=t^{-a}e^{-2\theta t}\int_0^t\int_0^te^{\theta s}e^{\theta r}m(s,r)drds\longrightarrow\frac{\Gamma(b+1)}{\theta^{b+3}},\label{cv2-m}
\end{eqnarray}
 where $\Gamma(.)$ is the standard gamma function, whereas the function $m(t,s)$ is defined in (\ref{function-m}).
\end{lemma}
\begin{proof}
We first prove (\ref{cv1-m}). We have,
\begin{eqnarray*}t^{-a}e^{-\theta t}\int_0^te^{\theta
s}m(t,s)ds&=&t^{-a}e^{-\theta t}\int_0^te^{\theta s}\int_{s}^{t}
u^a(t-u)^bduds\\
&=&t^{-a}e^{-\theta t}\int_0^tduu^a(t-u)^b\int_{0}^{u}dse^{\theta s}\\
&=&t^{-a}e^{-\theta t}\int_0^tduu^a(t-u)^b \frac{(e^{\theta
u}-1)}{\theta}
\\
&=&\frac{t^{-a}e^{-\theta t}}{\theta}\int_0^tu^a(t-u)^b  e^{\theta
u} du-\frac{t^{-a}e^{-\theta t}}{\theta}\int_0^tu^a(t-u)^b du.
\end{eqnarray*}
On the other hand, by the change of variables $x=t-u$, we get
\begin{eqnarray*} \frac{t^{-a}e^{-\theta t}}{\theta}\int_0^tu^a(t-u)^b
e^{\theta u} du &=&\frac{t^{-a}}{\theta}\int_0^t(t-x)^ax^b e^{-\theta x} dx\\
&=&\frac{1}{\theta}\int_0^t\left(1-\frac{x}{t}\right)^ax^b e^{-\theta x} dx\\
&\longrightarrow& \frac{1}{\theta}\int_0^{\infty}x^b e^{-\theta x}
dx=\frac{\Gamma(b+1)}{\theta^{b+2}}
\end{eqnarray*}as $t\longrightarrow\infty$. Moreover, by the change of variables
$x=u/t$,
\begin{eqnarray*} \frac{t^{-a}e^{-\theta t}}{\theta}\int_0^tu^a(t-u)^b
 du &=&\frac{e^{-\theta t}}{\theta}t^{b}\int_0^t(u/t)^a(1-\frac{u}{t})^b  dx\\
&=&\frac{e^{-\theta t}}{\theta}t^{b+1}\int_0^tx^a(1-x)^b  dx\\
&=&\frac{e^{-\theta t}}{\theta}t^{b+1}\beta(a+1,b+1)\\
&\longrightarrow& 0
\end{eqnarray*}as $t\longrightarrow\infty$. Thus the proof of the
convergence (\ref{cv1-m}) is done.\\
For (\ref{cv2-m}), using  L'H\^{o}pital's rule, we obtain
\begin{eqnarray*}\lim_{t\rightarrow\infty}t^{-a}e^{-2\theta t}\int_0^t\int_0^te^{\theta s}e^{\theta r}m(s,r)drds&=&
\lim_{t\rightarrow\infty}\frac{2\int_0^t\int_0^se^{\theta
s}e^{\theta r}m(s,r)drds}{t^{a}e^{2\theta
t}}\\
&=& \lim_{t\rightarrow\infty}\frac{2\int_0^t e^{\theta t}e^{\theta
r}m(t,r)dr }{t^{a}e^{2\theta t}\left(2\theta+\frac{a}{t}\right)}
\\
&=& \lim_{t\rightarrow\infty}\frac{2
}{\left(2\theta+\frac{a}{t}\right)}t^{-a}e^{-\theta
t}\int_0^te^{\theta r}m(t,r)dr\\&=&\frac{\Gamma(b+1)}{\theta^{b+3}},
\end{eqnarray*}
where  the latter equality comes from (\ref{cv1-m}). Therefore the
convergence (\ref{cv2-m}) is proved.
\end{proof}
\section{LSE  based on continuous-time
observation} In this section we will establish the consistency and
the asymptotic distribution of the least square-type estimator
$\widetilde{\theta}_t$ given in (\ref{estimator-cont}), based on the
continuous-time observation  $\{X_s,\ s\in[0,t]\}$ given by
(\ref{wfBm}), as $t\rightarrow\infty$.

Recall that if $X\sim \mathcal{N}(m_1,\sigma_1)$ and $Y\sim
\mathcal{N}(m_2,\sigma_2)$ are two independent random variables,
then $X/Y$ follows a   Cauchy-type distribution. For a motivation
and further references, we refer the reader to \cite{PTM}, as well
as \cite{marsaglia}. Notice also that if $N\sim\mathcal{N}(0,1)$ is
independent of $B^{a,b}$, then $N$ is independent of $Z_{\infty}$,
since $Z_{\infty}:=\int_0^\infty e^{-\theta s}B^{a,b}_sds$ is a
functional of $B^{a,b}$.

\begin{theorem}\label{asym-wfOU-cont-obse}
Assume that  $a>-1$, $|b|<1$, $|b|<a+1$, and let
$\widetilde{\theta_t}$ be the estimator given in
(\ref{estimator-cont}). Then, as $t\longrightarrow\infty$,
\begin{eqnarray*}\widetilde{\theta}_t\longrightarrow \theta \mbox{ almost surely}.
 \end{eqnarray*}
Moreover,  as $t\rightarrow \infty$,
\begin{eqnarray*}t^{-a/2}e^{\theta t}\left(\widetilde{\theta}_t-\theta\right)\overset{\texttt{law}}{\longrightarrow}
\frac{2\sigma_{B^{a,b}}}{\sqrt{E\left(Z_\infty^2\right)}}
\mathcal{C}(1),
\end{eqnarray*}where $\sigma_{B^{a,b}}=\frac{\Gamma(b+1)}{\theta^{b+1}}$, $Z_{\infty}:=\int_0^\infty
e^{-\theta s}B^{a,b}_sds$, whereas  $\mathcal{C}(1)$ is the standard
Cauchy distribution  with the probability density function
$\frac{1}{\pi (1+x^2)};\  x\in \mathbb{R}$.
\end{theorem}
\begin{proof} In order to prove this Theorem
\ref{asym-wfOU-cont-obse}, using Theorem \ref{asym-G-cont-obse}, it
suffices to check that the assumptions $(\mathcal{H}1)$,
$(\mathcal{H}2)$,  $(\mathcal{H}3)$, $(\mathcal{H}4)$ hold.\\
It follows from (\ref{decomp-main-R}) that for every $0<s\leq t$,
\begin{eqnarray*}E\left(B_t^{a,b}-B_s^{a,b}\right)^2&=&2\int_s^t
u^a(t-u)^b du\\&=&2(t-s)^{b+1}\int_0^1 (t(1-x)+sx)^a x^b dx,
\end{eqnarray*} where the latter inequality comes from the change of
variables $x=(t-u)/(t-s)$.\\
Hence, it is easy to see that there exists a constant $C_{a,b}$
depending only on $a,b$, such that
\begin{eqnarray*}E\left(B_t^{a,b}-B_s^{a,b}\right)^2\leq C_{a,b}
t^a(t-s)^{b+1}.
\end{eqnarray*}
Thus for all fixed $T$ there exists a constant $C_{a,b}(T)$
depending only on $a,b,T$ such that for every $0<s\leq t\leq T$
\begin{eqnarray*}E\left(B_t^{a,b}-B_s^{a,b}\right)^2\leq C_{a,b}(T)
|t-s|^{(a+b+1)\wedge(b+1)},
\end{eqnarray*} where we used the fact that $a+b+1>0$.\\
Therefore, using the fact that $B^{a,b}$ is Gaussian, and
Kolmogorov's continuity criterion, we deduce that $B^{a,b}$ has a
version with $((a+b+1)\wedge(b+1)-\varepsilon)$-H\"older continuous
paths for every $\varepsilon\in(0,(a+b+1)\wedge(b+1))$. Thus
$(\mathcal{H}1)$ holds for any $\delta$ in
$(0,(a+b+1)\wedge(b+1))$.\\ On the other hand, according to
(\ref{decomp-main-R}) we have for every $t\geq0$,
\[E\left(B_t^{a,b}\right)^2=2\beta(1+a,1+b)t^{a+b+1},\] which proves
that $(\mathcal{H}2)$ holds for $\gamma=(a+b+1)/2$.\\
 Now  it remains to check that the assumptions $(\mathcal{H}3)$ and
$(\mathcal{H}4)$ hold for $\nu=-a/2$ and
$\sigma_{B^{a,b}}=\frac{\Gamma(b+1)}{\theta^{b+1}}$.
 Let us first compute the
limiting variance of $t^{-a/2}e^{-\theta t}\int_0^te^{\theta
s}dB^{a,b}_s$ as $t\rightarrow\infty$. By (\ref{decomp-main-R}) we
obtain
\begin{eqnarray}&&E\left[\left(t^{-a/2}e^{-\theta t}\int_0^te^{\theta
s}dB^{a,b}_s\right)^2\right]= E\left[\left(t^{-a/2}e^{-\theta
t}\left(e^{\theta t}B^{a,b}_t-\theta\int_0^te^{\theta
s}B^{a,b}_sds\right)\right)^2\right]\nonumber\\&=&
t^{-a}\left(R^{a,b}(t,t)-2\theta e^{-\theta t}\int_0^te^{\theta s}
R^{a,b}(t,s)ds+\theta^2e^{-2\theta t}\int_0^t\int_0^te^{\theta
s}e^{\theta r}R^{a,b}(s,r)dsdr\right)\nonumber\\&=&
t^{-a}\Delta_{g_{B^{a,b}}}(t)+2\theta I_{t}- \theta^2
J_{t},\label{Delta+I+J fBm}
\end{eqnarray}where $I_{t},\ J_{t}$ and $\Delta_{g_{B^{a,b}}}(t)$
are defined in (\ref{cv1-m}), (\ref{cv2-m}) and Lemma \ref{lemma
key1 of applications}, respectively,
whereas $g_{B^{a,b}}(s,r)=\beta(a+1,b+1)\left(s^{a+b+1}+r^{a+b+1}\right)$.\\
On the other hand,  since $\frac{\partial g_{B^{a,b}}}{\partial
s}(s,0)=\beta(a+1,b+1)(a+b+1)s^{a+b}$, and $\frac{\partial^2
g_{B^{a,b}} }{\partial s\partial r}(s,r)=0$, it follows from
(\ref{key1})  that
\begin{eqnarray}
t^{-a}\Delta_{g_{B^{a,b}}}(t)&=&2\beta(a+1,b+1)(a+b+1)t^{-a}e^{-2\theta
t}\int_0^ts^{a+b}e^{\theta s}ds\nonumber\\
&\leq&2\beta(a+1,b+1)e^{-\theta t}t^{a+b+1}\nonumber
\\&&\longrightarrow0 \mbox{ as } t\rightarrow\infty.\label{Delta
wfBm}
\end{eqnarray}
Combining (\ref{Delta+I+J fBm}), (\ref{Delta wfBm}), (\ref{cv1-m})
and (\ref{cv2-m}), we get
\begin{eqnarray*}E\left[\left(t^{-a/2}e^{-\theta t}\int_0^te^{\theta
s}dB^{a,b}_s\right)^2\right]&\longrightarrow&\frac{\Gamma(b+1)}{\theta^{b+1}}\quad\mbox{
as $t\rightarrow\infty$},
\end{eqnarray*} which implies that $(\mathcal{H}3)$ holds.\\
Hence, to finish the proof it remains to check that $(\mathcal{H}4)$
holds, that is, for all fixed $s\geq0$
\begin{eqnarray*} \label{cv Bs int}\lim_{t\rightarrow\infty}E\left(B^{a,b}_st^{-a/2}e^{-\theta t}\int_0^te^{\theta
r}dB^{a,b}_r\right)=0.
\end{eqnarray*}
Let us consider $s<t$.   According to   (\ref{Integration by
parts}), we can write
\begin{eqnarray*} &&E\left(B^{a,b}_st^{-a/2}e^{-\theta t}\int_0^te^{\theta
r}dB^{a,b}_r\right) \\&=&t^{-a/2}\left(R^{a,b}(s,t)- \theta
e^{-\theta t}\int_0^te^{\theta r}R^{a,b}(s,r)dr\right)
\\&=&t^{-a/2}\left(R^{a,b}(s,t)- \theta e^{-\theta
t}\int_s^te^{\theta r}R^{a,b}(s,r)dr- \theta e^{-\theta
t}\int_0^se^{\theta r}R^{a,b}(s,r)dr\right)
\\&=&  t^{-a/2}\left(e^{-\theta (t-s)}R^{a,b}(s,s)+  e^{-\theta t}\int_s^te^{\theta r}
\frac{\partial R^{a,b}}{\partial r}(s,r)dr- \theta e^{-\theta
t}\int_0^se^{\theta r}R^{a,b}(s,r)dr\right).
\end{eqnarray*}
It is clear that $t^{-a/2}\left(e^{-\theta (t-s)}R^{a,b}(s,s)-
\theta e^{-\theta t}\int_0^se^{\theta
r}R^{a,b}(s,r)dr\right)\longrightarrow0$
 as $t\rightarrow\infty$. Let us now prove that  \[t^{-a/2}e^{-\theta t}\int_s^te^{\theta r}
\frac{\partial R^{a,b}}{\partial r}(s,r)dr\longrightarrow0\]
 as $t\rightarrow\infty$. Using (\ref{cov-B}) we have for $s<r$
 \[\frac{\partial R^{a,b}}{\partial r}(s,r)=b\int_0^{s} u^a (r-u)^{b-1}du\]
Applying L'H\^ospital's rule we obtain
\begin{eqnarray*}
\lim_{t\rightarrow\infty}t^{-a/2}e^{-\theta t}\int_s^te^{\theta r}
\frac{\partial R^{a,b}}{\partial r}(s,r)dr
&=&\lim_{t\rightarrow\infty}\frac{bt^{-a/2}}{\theta+\frac{a}{2t}}\int_0^{s}
u^a (t-u)^{b-1}du
\nonumber\\&=&\lim_{t\rightarrow\infty}\frac{bt^{b-1-\frac{a}{2}}}{\theta+\frac{a}{2t}}\int_0^{s}
u^a (1-u/t)^{b-1}du \nonumber\\&\longrightarrow&0 \mbox{ as }
t\rightarrow\infty,
\end{eqnarray*}
due to  $b-1-\frac{a}{2}<0$. In fact, if $-1<a<0$, we use $b<a+1$,
then $b<a+1<\frac{a}{2}+1$. Otherwise, if $a>0$, we use $b<1$, then
$b-1-\frac{a}{2}<b-1-<0$.
 Therefore the proof of Theorem
\ref{asym-wfOU-cont-obse} is complete.

\end{proof}

\section{LSEs  based on discrete-time observations}
In this section, our purpose is to study the asymptotic behavior and
the rate consistency of the estimators $\hat{\theta}_{n}$ and
$\check{\theta}_n$ based on the sampling data $X_{t_i},
i=0,\ldots,n$ of (\ref{wfBm}), where $t_i=i\Delta_{n},
i=0,\ldots,n$, and $T_n=n\Delta_{n}$ denotes the length of the
``observation window".

\begin{definition}
Let $\{Z_n\}$ be a sequence of random variables defined on a
probability space $(\Omega, \mathcal{F},P)$. We say $\{Z_n\}$ is
tight (or bounded in probability), if for every $\varepsilon>0$,
there exists $M_{\varepsilon} > 0$ such that, \[P\left(|Z_n| >
M_{\varepsilon}\right) < \varepsilon,\quad \mbox{for all } n.\]
\end{definition}

\subsection{The asymptotic behavior and the rate consistency of LSEs}
\begin{theorem}\label{hat-check-wfOU-asymp-thm}
 Assume that  $a>-1$, $|b|<1$, $|b|<a+1$. Let $\hat{\theta}_n$ and $\check{\theta}_n$
be the estimators given in (\ref{expression thetahat}) and
(\ref{expression thetaCheck}), respectively. Suppose that
$\Delta_n\rightarrow0$ and $n\Delta_n^{1+\alpha}\rightarrow\infty$
for some $\alpha>0$. Then, as $n\rightarrow\infty,$
\begin{eqnarray*}\label{cv-ps-theta-hat}\hat{\theta}_n\longrightarrow\theta,
\quad\check{\theta}_n\longrightarrow\theta\quad \mbox{almost
surely,}
\end{eqnarray*} and for any $q\geq 0,$
\begin{eqnarray*}
\Delta_n^qe^{\theta T_n}(\hat{\theta}_{n}-\theta)\mbox{  and }
\Delta_n^qe^{\theta T_n}(\check{\theta}_{n}-\theta)\mbox{ are
  not   tight}.
\end{eqnarray*}
In addition, if we assume that $n\Delta_n^{3}\rightarrow0$ as
 $n\rightarrow\infty$, the estimators $\hat{\theta}_{n}$ and
 $\check{\theta}_{n}$ are
$\sqrt{T_n}-consistent$ in  the sense that the sequences
\[\sqrt{T_n}(\hat{\theta}_{n}-\theta) \mbox{ and }
\sqrt{T_n}(\check{\theta}_{n}-\theta) \mbox{ are
     tight.}\]
\end{theorem}
\begin{proof} In order to prove this Theorem
\ref{hat-check-wfOU-asymp-thm}, using Theorem
\ref{hat-check-asymp-thm}, it suffices to check that the assumptions
$(\mathcal{H}1)$, $(\mathcal{H}2)$,  $(\mathcal{H}5)$ hold.\\
 From the proof of Theorem
\ref{asym-wfOU-cont-obse}, the assumptions $(\mathcal{H}1)$,
$(\mathcal{H}2)$ hold. Now it remains to check that $(\mathcal{H}5)$
holds. In this case, the process $\zeta$ is defined as
\[\zeta_t:=\int_0^te^{-\theta s}dB^{a,b}_s,\qquad t\geq0,\]
whereas the  integral is interpreted in the Young sense (see
Appendix).\\
Using the formula (\ref{Integration by parts}) and (\ref{key2}), we
can write
\begin{eqnarray*}
E\left[\left(\zeta_{t_{i}}-\zeta_{t_{i-1}}\right)^2\right]&=&E\left[\left(\int_{t_{i-1}}^{t_{i}}e^{-\theta
s}dB^{a,b}_s\right)^2\right]\\
&=&E\left[\left(e^{-\theta t_{i}}B_{t_{i}}^{a,b}-e^{-\theta
t_{i-1}}B_{t_{i-1}}^{a,b}+\theta\int_{t_{i-1}}^{t_{i}}e^{-\theta
s}B^{a,b}_sds\right)^2\right]
\\&=&\lambda_{g_{B^{a,b}}}(t_{i},t_{i-1})-\lambda_{m}(t_{i},t_{i-1})
\\&=&\int_{t_{i-1}}^{t_{i}}\int_{t_{i-1}}^{t_{i}} e^{-\theta (r+u)}\frac{\partial^2 g_{B^{a,b}}}{\partial
r\partial u}(r,u)dr du-\lambda_{m}(t_{i},t_{i-1})
\\&=&-\lambda_{m}(t_{i},t_{i-1}),
\end{eqnarray*} where $\lambda_{.}(t_{i},t_{i-1})$ is defined in Lemma \ref{lemma key2 of applications},
$g_{B^{a,b}}(s,r)=\beta(a+1,b+1)\left(s^{a+b+1}+r^{a+b+1}\right)$
and $\frac{\partial^2 g_{B^{a,b}} }{\partial s\partial r}(s,r)=0$,
whereas the term $\lambda_{m}(t_{i},t_{i-1})$ is equal to
\begin{eqnarray*}\lambda_{m}(t_{i},t_{i-1})&=&-2m(t_{i},t_{i-1})e^{-2\theta (t_{i-1}+t_{i})}+2\theta e^{-\theta t_{i}}\int_{t_{i-1}}^{t_{i}}
m(r,{t_{i}})e^{-\theta r}dr \\&&-2\theta e^{-\theta
{t_{i-1}}}\int_{t_{i-1}}^{t_{i}} m(r,{t_{i-1}})e^{-\theta r}dr
+\theta^2 \int_{t_{i-1}}^{t_{i}}\int_{t_{i-1}}^{t_{i}}
m(r,u)e^{-\theta (r+u)}dr du .
\end{eqnarray*}
Combining this with the fact  for every $t_{i-1}\leq u\leq r\leq
t_{i}$, $i\geq2$,
\begin{eqnarray*}
|m(r,u)| &=&\left|\int_u^rx^a(r-x)^bdx\right|\\&\leq &     \left\{
\begin{array}{ll}
\left|r^a\int_u^r(r-x)^bdx\right| & \mbox{ if }-1<a <0 \\
\\
\left|u^a\int_u^r(r-x)^bdx\right| & \mbox{ if }a>0
\end{array}%
\right.
\\&\leq &     \left\{
\begin{array}{ll}
\frac{\Delta_n^{a+b+1}}{b+1} & \mbox{ if }-1<a \leq0 \\
\\
\frac{(n\Delta_n)^a\Delta_n^{b+1}}{b+1} & \mbox{ if }a>0
\end{array}%
\right.
\end{eqnarray*}
together with $\Delta_n\longrightarrow0$, we deduce that there is a
positive constant $C$ such that
\begin{eqnarray*}
E\left[\left(\zeta_{t_{i}}-\zeta_{t_{i-1}}\right)^2\right]&\leq & C
\left\{
\begin{array}{ll}
\frac{\Delta_n^{a+b+1}}{b+1} & \mbox{ if }-1<a \leq0 \\
\\
\frac{(n\Delta_n)^a\Delta_n^{b+1}}{b+1} & \mbox{ if }a>0,
\end{array}%
\right.
\end{eqnarray*}
which proves that the assumption $(\mathcal{H}5)$ holds. Therefore
the desired result is obtained.
\end{proof}

\subsection{Numerical Results}
 Here we simulate 100 sample paths of the process $X$, given by
(\ref{wfBm}), using a regular partition of $n=2000$ intervals. The
tables below report the means and standard deviations of the
proposed estimators   $\hat{\theta}_{n}$ and $\check{\theta}_n$
defined, respectively,  by (\ref{expression thetahat}) and
(\ref{expression thetaCheck})  of the true value of the
 parameter $\theta$. The tables   confirm that the estimators $\hat{\theta}_{n}$ and $\check{\theta}_n$ are strongly
consistent even for small values of $n$ and have small standard
deviations for different true values of $\theta$.\\

\begin{table}[H]
    \centering
    \caption{The means and standard deviations of the estimator $\hat{\theta}_n$ for $a=0.5$ and $b=0.9$.}
    \label{t-theta2}
    \begin{tabular}{|c|c|c|}
         \cline{1-3}
        &$\theta=0.7$&$\theta=0.9$ \\\cline{1-3}
        Mean&0.7223197&0.9075169 \\\cline{1-3}
        Median&0.7286532&0.9131851 \\\cline{1-3}
        Std. dev.&0.1066532&0.08772779 \\\cline{1-3}
    \end{tabular}
\end{table}
\begin{table}[H]
    \centering
    \caption{The means and standard deviations of the estimator $\hat{\theta}_n$ for $a=0.1$ and $b=0.4$.}
    \label{t-theta4}
    \begin{tabular}{|c|c|c|}
        \cline{1-3}
        &$\theta=0.7$&$\theta=0.9$ \\\cline{1-3}
        Mean &0.6764275 &0.8892152\\\cline{1-3}
        Median &0.7006374 &0.9028244\\\cline{1-3}
        Std. dev. &0.1260278 &0.08042517\\\cline{1-3}
    \end{tabular}
\end{table}

\begin{table}[H]
        \centering
        \caption{The means and standard deviations of the estimator $\check{\theta}_n$ for $a=0.5$ and $b=0.9$.}
        \label{t-theta}
        \begin{tabular}{|c|c|c|}
            \cline{1-3}
            &$\theta=0.7$&$\theta=0.9$\\\cline{1-3}
            Mean&0.7234658&0.9091066 \\\cline{1-3}
            Median&0.7294589&0.9144367 \\\cline{1-3}
            Std. dev.&0.1064266&0.0879878 \\\cline{1-3}
        \end{tabular}
    \end{table}

\begin{table}[H]
    \centering
    \caption{The means and standard deviations of the estimator $\check{\theta}_n$ for $a=0.1$ and $b=0.4$.}
    \label{t-theta3}
    \begin{tabular}{|c|c|c|}
        \cline{1-3}
        &$\theta=0.7$&$\theta=0.9$\\\cline{1-3}
        Mean&0.6802641 &0.8911407\\\cline{1-3}
        Median&0.7015955 &0.9041038\\\cline{1-3}
        Std. dev.&0.120526 &0.07755756\\\cline{1-3}
    \end{tabular}
\end{table}

\section{Conclusion}
To conclude, in this paper we provide least squares-type estimators
for the drift parameter $\theta$ of  the weighted fractional
Ornstein-Uhlenbeck process $X$, given by (\ref{wfBm}), based
continuous-time and discrete-time observations of  $X$. The novelty
of our approach is that it allows, comparing with the literature on
statistical inference for $X$ discussed in \cite{SYY,SYY2,CSC}, to
consider the general case $a>-1$, $|b|<1$ and $|b|<a+1$. More
precisely,
\begin{itemize}
\item We estimate the drift parameter $\theta$ of (\ref{wfBm}) based on the continuous-time observations
$\{X_s,\ s\in[0,t]\}$, as $t\rightarrow\infty$. We  prove the strong
consistency and the asymptotic behavior in distribution of the
estimator $\widetilde{\theta}_t$   for all parameters $a>-1$,
$|b|<1$ and $|b|<a+1$. Our results extend those proved in
\cite{SYY,SYY2}, where $-\frac12<a<0$,   $-a<b<a+1$ only.
\item Suppose that  the process $X$ given in (\ref{wfBm}) is observed equidistantly
in time with the step size $\Delta_n$: $t_i=i\Delta_{n},
i=0,\ldots,n$. We estimate the drift parameter $\theta$ of
(\ref{wfBm})  on the sampling data $X_{t_i}, i=0,\ldots,n$, as
$\Delta_n\longrightarrow0$ and $n\longrightarrow\infty$. We  study
the asymptotic behavior and the rate consistency of the estimators
$\hat{\theta}_{n}$ and $\check{\theta}_n$ for all parameters $a>-1$,
$|b|<1$ and $|b|<a+1$. In this case, our results extend those proved
in \cite{CSC}, where $-1<a<0$, $-a<b<a+1$ only.
\end{itemize}
The proofs of the asymptotic behavior of the estimators
 are based on a new decomposition of the covariance function  $R^{a,b}(t,s)$ of the wfBm $B^{a,b}$ (see Lemma \ref{lemma-cov-R}),
  and slight extensions of  results
\cite{EEO}  and \cite{EAA} (see Theorem \ref{asym-G-cont-obse} and Theorem \ref{hat-check-asymp-thm} in Appendix).\\

\noindent {\bf Acknowledgments}\\

\noindent The authors would like to thank the  two anonymous
referees for their careful reading of the manuscript and for their
valuable suggestions and remarks.

\section{Appendix}

Here we present some ingredients needed in the paper.

 Let $G = \left(G_t,t\geq0\right)$  be a continuous centered Gaussian process defined on some probability
 space $(\Omega, \mathcal{F}, P)$ (Here,
and throughout the text, we assume that $\mathcal{F}$ is the
sigma-field generated by $G$). In this section  we consider the
non-ergodic case of Gaussian Ornstein-Uhlenbeck processes
$X=\left\{X_t, t\geq0\right\}$ given by the following linear
stochastic differential equation
\begin{eqnarray}\label{GOU}X_0=0;\quad  dX_t=\theta X_tdt+dG_t,\quad t\geq0,
\end{eqnarray}where   $\theta>0$   is an unknown parameter.
It is clear that the linear equation (\ref{GOU}) has the following
explicit solution
\begin{eqnarray*}X_t=e^{\theta t}\zeta_t,\qquad t\geq0,
\end{eqnarray*}
where
\[\zeta_t:=\int_0^te^{-\theta s}dG_s,\qquad t\geq0,\]
whereas this latter  integral is interpreted in the Young sense.\\

Let us introduce the following required assumptions.
\begin{itemize}
\item[$(\mathcal{H}1)$] The process $G$   has H\"older continuous paths of some order $\delta\in(0,1]$.
\item[$(\mathcal{H}2)$] For every $t\geq0$, $E\left(G_t^2\right)\leq ct^{2\gamma}$ for some positive constants  $c$ and $ \gamma$.
\item[$(\mathcal{H}3)$] There is constant $\nu$ in $\mathbb{R}$ such that the limiting variance of $t^{\nu}e^{-\theta t}\int_0^te^{\theta s}dG_s$  exists as
$t\rightarrow\infty$, that is, there exists a constant $\sigma_G>0$
such that
\[\lim_{t\rightarrow\infty}E\left[\left(t^{\nu}e^{-\theta t}\int_0^te^{\theta s}dG_s\right)^2\right]=\sigma_G^2.\]
\item[$(\mathcal{H}4)$] For   $\nu$  given in $(\mathcal{H}3)$, we have all fixed $s\geq0$
\begin{eqnarray*} \lim_{t\rightarrow\infty}E\left(G_st^{\nu}e^{-\theta t}\int_0^te^{\theta
r}dG_r\right)=0.
\end{eqnarray*}
\item[$(\mathcal{H}5)$]There exist positive constants $\rho,C$ and a real constant $\mu$ such that
\[E\left[\left(\zeta_{t_{i}}-\zeta_{t_{i-1}}\right)^2\right]\leq
C (n\Delta_n)^{\mu}\Delta_n^{\rho}
 e^{-2\theta t_{i}}\ \mbox{ for every $i=1,\ldots, n$, $n\geq1$}.\]
\end{itemize}

The following theorem  is a slight extension of the main result in
\cite{EEO}, and it can be established following the same arguments
as in \cite{EEO}.

\begin{theorem}\label{asym-G-cont-obse}
Assume that $(\mathcal{H}1)$ and $(\mathcal{H}2)$  hold and let
$\widetilde{\theta_t}$ be the estimator of the form
(\ref{estimator-cont}). Then, as $t\longrightarrow\infty$,
\begin{eqnarray*}\widetilde{\theta}_t\longrightarrow \theta \mbox{ almost surely}.
 \end{eqnarray*}
Moreover, if $(\mathcal{H}1)$-$(\mathcal{H}4)$  hold, then, as
$t\rightarrow \infty$,
\begin{eqnarray*}t^{\nu}e^{\theta t}\left(\widetilde{\theta}_t-\theta\right)\overset{\texttt{law}}{\longrightarrow}\frac{2\sigma_G}{\sqrt{E\left(Z_\infty^2\right)}}
\mathcal{C}(1),
\end{eqnarray*}where $Z_{\infty}:=\int_0^\infty
e^{-\theta s}G_sds$, whereas  $\mathcal{C}(1)$ is the standard
Cauchy distribution  with the probability density function
$\frac{1}{\pi (1+x^2)};\  x\in \mathbb{R}$.
\end{theorem}

The following theorem  is also a slight extension of the main result
in \cite{EAA}, and it can be proved following line by line the
proofs given in \cite{EAA}.

\begin{theorem}\label{hat-check-asymp-thm}
 Assume that $(\mathcal{H}1)$, $(\mathcal{H}2)$
and $(\mathcal{H}5)$  hold. Let $\hat{\theta}_n$ and
$\check{\theta}_n$ be the estimators of the forms (\ref{expression
thetahat}) and (\ref{expression thetaCheck}), respectively. Suppose
that $\Delta_n\rightarrow0$ and
$n\Delta_n^{1+\alpha}\rightarrow\infty$ for some $\alpha>0$. Then,
as $n\rightarrow\infty,$
\begin{eqnarray*}\label{cv-ps-theta-hat}\hat{\theta}_n\longrightarrow\theta,
\quad\check{\theta}_n\longrightarrow\theta\quad \mbox{almost
surely,}
\end{eqnarray*} and for any $q\geq 0,$
\begin{eqnarray*}
\Delta_n^qe^{\theta T_n}(\hat{\theta}_{n}-\theta)\mbox{  and }
\Delta_n^qe^{\theta T_n}(\check{\theta}_{n}-\theta)\mbox{ are
  not   tight}.
\end{eqnarray*}
In addition, if we assume that $n\Delta_n^{3}\rightarrow0$ as
 $n\rightarrow\infty$, the estimators $\hat{\theta}_{n}$ and
 $\check{\theta}_{n}$ are
$\sqrt{T_n}-consistent$ in  the sense that the sequences
\[\sqrt{T_n}(\hat{\theta}_{n}-\theta) \mbox{ and }
\sqrt{T_n}(\check{\theta}_{n}-\theta) \mbox{ are
     tight.}\]
\end{theorem}

\begin{lemma}[\cite{EEO}]\label{lemma key1 of applications}Let $g : [0,\infty)\times[0,\infty)\longrightarrow\mathbb{R}$
 be a symmetric function such that $\frac{\partial  g }{\partial s }(s,r)$
 and $\frac{\partial^2 g }{\partial s\partial r}(s,r)$ integrable on   $(0,\infty)\times[0,\infty)$. Then, for every $t\geq0$,
\begin{eqnarray}  \Delta_g(t)&:=&g(t,t)-2\theta e^{-\theta t}\int_0^t g(s,t)e^{\theta s}ds
+\theta^2 e^{-2\theta t}\int_0^t\int_0^t g(s,r)e^{\theta (s+r)}dr ds
\nonumber\\&=&2e^{-2\theta t}\int_0^te^{\theta s} \frac{\partial g
}{\partial s}(s,0)ds+2 e^{-2\theta t}\int_0^tdse^{\theta s}\int_0^s
dr\frac{\partial^2 g }{\partial s\partial r}(s,r)e^{\theta
r}.\label{key1}
\end{eqnarray}
\end{lemma}
\begin{lemma}[\cite{EAA}]\label{lemma key2 of applications}Let $g : [0,\infty)\times[0,\infty)\longrightarrow\mathbb{R}$
 be a symmetric function such that $\frac{\partial  g }{\partial s }(s,r)$
 and $\frac{\partial^2 g }{\partial s\partial r}(s,r)$ integrable on   $(0,\infty)\times[0,\infty)$.
 Then, for every  $t\geq s\geq0$,
\begin{eqnarray}  \lambda_g(t,s)&:=&g(t,t)e^{-2\theta t}+g(s,s)e^{-2\theta
s}-2g(s,t)e^{-2\theta (s+t)}+2\theta e^{-\theta t}\int_s^t
g(r,t)e^{-\theta r}dr \nonumber\\&&-2\theta e^{-\theta s}\int_s^t
g(r,s)e^{-\theta r}dr +\theta^2 \int_s^t\int_s^t g(r,u)e^{-\theta
(r+u)}dr du \nonumber\\&=&\int_s^t\int_s^t e^{-\theta
(r+u)}\frac{\partial^2 g }{\partial r\partial u}(r,u)dr
du.\label{key2}
\end{eqnarray}
\end{lemma}

Let us now recall the Young integral introduced in \cite{Young}. For
any $\alpha\in (0,1]$, we denote by $\mathcal{H}^\alpha([0,T])$ the
set of $\alpha$-H\"older continuous functions, that is, the set of
functions $f:[0,T]\to\mathbb{R}$ such that
\[
|f|_\alpha := \sup_{0\leq s<t\leq
T}\frac{|f(t)-f(s)|}{(t-s)^{\alpha}}<\infty.
\]
We also set $|f|_\infty=\sup_{t\in[0,T]}|f(t)|$, and we equip
$\mathcal{H}^\alpha([0,T])$ with the norm $\|f\|_\alpha :=
|f|_\alpha + |f|_\infty.$\\
 Let $f\in\mathcal{H}^\alpha([0,T])$, and
consider the operator $T_f:\mathcal{C}^1([0,T])
\to\mathcal{C}^0([0,T])$ defined as
\[
T_f(g)(t)=\int_0^t f(u)g'(u)du, \quad t\in[0,T].
\]
It can be shown (see, e.g., \cite[Section 3.1]{nourdin}) that, for
any $\beta\in(1-\alpha,1)$, there exists a constant
$C_{\alpha,\beta,T}>0$ depending only on $\alpha$, $\beta$ and $T$
such that, for any $g\in\mathcal{H}^\beta([0,T])$,
\[
\left\|\int_0^\cdot f(u)g'(u)du\right\|_\beta \leq
C_{\alpha,\beta,T} \|f\|_\alpha \|g\|_\beta.
\]
We deduce that, for any $\alpha\in (0,1)$, any
$f\in\mathcal{H}^\alpha([0,T])$ and any $\beta\in(1-\alpha,1)$, the
linear operator
$T_f:\mathcal{C}^1([0,T])\subset\mathcal{H}^\beta([0,T])\to
\mathcal{H}^\beta([0,T])$, defined as $T_f(g)=\int_0^\cdot
f(u)g'(u)du$, is continuous with respect to the norm
$\|\cdot\|_\beta$. By density, it extends (in an unique way) to an
operator defined on $\mathcal{H}^\beta$. As consequence, if
$f\in\mathcal{H}^\alpha([0,T])$, if $g\in\mathcal{H}^\beta([0,T])$
and if $\alpha+\beta>1$, then the (so-called) Young integral
$\int_0^\cdot f(u)dg(u)$ is well-defined as being $T_f(g)$ (see
\cite{Young}).

The Young integral obeys the following formula. Let
$f\in\mathcal{H}^\alpha([0,T])$ with $\alpha\in(0,1)$ and
$g\in\mathcal{H}^\beta([0,T])$ with $\beta\in(0,1)$ such that
$\alpha+\beta>1$. Then $\int_0^. g_udf_u$ and $\int_0^. f_u dg_u$
are well-defined as the Young integrals. Moreover, for all
$t\in[0,T]$,
\begin{eqnarray}\label{Integration by parts}
f_tg_t=f_0g_0+\int_0^t g_udf_u+\int_0^t f_u dg_u.
\end{eqnarray}

\end{document}